\documentclass[11pt]{article}

\usepackage{amsmath,amssymb}

\usepackage{amssymb}
\usepackage{amscd}
\usepackage{amsmath}
\usepackage{amsfonts}
\usepackage{amsthm}

\newtheorem{theo}{Theorem}[section]

\newtheorem{prop}[theo]{Proposition}
\newtheorem{lem}[theo]{Lemma}

\theoremstyle{remark}

\newcommand{\R}{{\mathbb{R}}}

\newcommand{\C}{{\mathbb{C}}}
\newcommand{\Z}{{\mathbb{Z}}}

\newcommand{\al}{\alpha} 
\newcommand{\be}{\beta} 
\newcommand{\Om}{\Omega} 
 
\newcommand{\ga}{\gamma}
 
\newcommand{\Si}{\Sigma} 
\newcommand{\si}{\sigma}

\newcommand{\op}{\operatorname}

\newcommand{\mo}{\mathcal{M}}
\newcommand{\moir}{\mathcal{M}^{0}} 

\newcommand{\Ci}{\mathcal{C}^{\infty}}

\newcommand{\Co}{\mathcal{C}}
\newcommand{\rep}{\mathcal{R}}
\newcommand{\repir}{\mathcal{R}^0}

\newcommand{\jac}{\mathcal{J}}
\newcommand{\uun}{\op{U}(1)}

\newcommand{\Addresses}{{
  \bigskip
  \footnotesize

  Laurent~Charles, \textsc{ Institut de Math\'{e}matiques de Jussieu-Paris Rive Gauche, Sorbonne Universit\'{e}s, UPMC Univ Paris 06, F-75005, Paris, France}\par\nopagebreak
  \textit{E-mail address}: \texttt{laurent.charles@imj-prg.fr}

  \medskip

  Lisa~Jeffrey, \textsc{ Mathematics Departement, University of Toronto, Toronto, ON, Canada M5S 2E4 }\par\nopagebreak
  \textit{E-mail address}: \texttt{jeffrey@math.toronto.edu}
}}

\title{Torsion and symplectic volume in Seifert manifolds} 
\author{L. Charles and L. Jeffrey\thanks{Supported in part by a grant from NSERC}}

\begin{document}

\maketitle

\begin{abstract}For any oriented Seifert manifold $X$ and compact connected Lie group $G$ with finite center, we relate the Reidemeister density of the moduli space of representations of the fundamental group of $X$ into $G$ to the Liouville measure of some moduli spaces of representations of surface groups into $G$. 
\end{abstract}

\section{Introduction}

For any Lie group $G$ and manifold $Y$, the moduli space $\mo (Y )$ of conjugacy classes of representations of $\pi_1(Y)$ in $G$, has natural differential geometric structures. If $\Si$ is a closed oriented surface, $\mo (\Si)$ has a symplectic stucture defined via intersection pairing \cite{AtBo}, \cite{Go}. More generally, if $\Si$ is a compact oriented surface and $u \in \mo ( \partial \Si)$, the  subspace  $\mo ( \Si ,u )$ of $\mo (\Si)$ consisting of the representations restricting to $u$ on the boundary has a natural symplectic structure.  If $X$ is a closed 3-dimensional oriented manifold, $\mo (X)$ has a natural density $\mu_X$ defined from Reidemeister torsion \cite{Wi}. 

In this article, we relate these structures for $X$ any oriented Seifert manifold  and $\Si$ a convenient oriented surface embedded in $X$. We will prove that when $G$ is compact with finite center, the subspace $\moir (X) \subset \mo (X)$ of irreducible representations, is a smooth manifold covered by disjoint open subsets $O_\al$, such that each $O_\al$ identifies with $\moir (\Si , u_\al)$ for some $u_\al \in \mo (\partial \Si)$. Furthermore, on each $U_\al$ the canonical density $\mu_X$ identifies, up to some multiplicative constant depending on $\al$, with the Liouville measure of the symplectic structure of $\moir ( \Si, u_{\al})$.

Our main motivation is the Witten's asymptotic conjecture, which predicts that the Witten-Reshetikhin-Turaev invariant of a 3-manifold $X$ has a precise asymptotic expansion in the large level limit. This expansion is a sum of oscilatory terms, whose amplitudes are function of the Reidemeister volume of the components of $\mo (X)$. In the case where $X$ is a Seifert manifold, some of these amplitudes are actually function of the symplectic volumes of the moduli spaces $\mo^0 ( \Si , u)$, \cite{Ro}, \cite{oim}. So a relation between Reidemeister and symplectic volumes was expected. At a more general level, it is known that the Chern-Simons theory on a Seifert manifold can be interpreted as two-dimensional Yang-Mills theory \cite{BeWi}.

Let us state our results with more detail and then discuss the related literature.

\subsection*{Statement of the main result} 
The Seifert manifolds we will consider are the oriented closed connected three manifold equipped with a locally free circle action. Any such manifold may be obtained as follows. 
Let $\Si$ be an oriented compact surface with $n \geqslant 1$ boundary components $C_1$, \ldots, $C_n$. Let $D $ be the standard closed disk of $\C$. Let $\varphi_i$ be an orientation reversing diffeomorphism from $\partial D \times S^1$ to $C_i \times S^1$. Let $X$ be the manifold obtained by gluing $n$ copies of $D \times S^1$ to $\Si \times S^1$  through the maps $\varphi_i$. We have $[\varphi_i (\partial D)] = - p_i [C_i] + q_i [S^1]$  in $H_1( C_i \times S^1)$ where $p_i$, $q_i$ are two relatively prime integers. We assume that $p_i \geqslant 1$ for all $i$. 

Let $G$ be a compact connected Lie group with finite center. For $Y = X$, $\Si$, $C_i$ or $S^1$, we denote by $\mo (Y)$ (resp. $\moir (Y)$) the set of representations (resp. irreducible representations) of $\pi_1 (Y)$ in $G$ up to conjugation. 
 Since $C_i$ and $S^1$ are oriented circles, we can identify $\mo (C_i)$ and $\mo (S^1)$ with the set $\Co (G)$ of conjugacy classes of $G$. For any $u \in \Co (G)^n$, we denote by $\moir ( \Si, u)$ the subset of $\moir (\Si)$ consisting of the representations whose restriction to each $C_i$ is $u_i$. Recall that $\moir ( \Si , u)$ is a smooth symplectic manifold. 

For any $(u,v) \in \Co (G)^{n+1}$, we denote $\moir (X, u, v)$ the subset of $\moir (X)$ consisting of representations whose restriction to each $C_i$ is $u_i$ and to $S^1$ is $v$. Let $\mathcal{P}$ be the subset of $\Co (G)^{n+1}$ consisting of the $(u,v)$ such that $\moir (X, u,v)$ is non empty.

\begin{theo}\label{theo:smooth}
 $\moir (X)$ is a smooth manifold, whose components may have different dimensions. For any $[\rho] \in \moir (X)$, the tangent space $T_{[\rho]} \moir (X)$ is canonically identified with $H^1 (X, \op{Ad} \rho)$ where $\op{Ad} \rho$ is the flat vector bundle associated to $\rho$ via the adjoint representation. Furthermore, $\mathcal{P}$ is finite and for any $(u, v) \in  \mathcal{P}$, $\moir (X, u,v)$ is an open subset of $\moir(X)$ and the restriction map $R_{u.v} $ from $\moir (X,u,v)$ to $\moir ( \Si, u )$ is a diffeomorphism.
\end{theo} 

For any irreducible representation $ \rho$ of $\pi_1 (X)$ in $G$, the homology groups $H_0 ( X, \op{Ad} \rho)$ and $H_3 ( X, \op{Ad} \rho )$ are trivial. By Poincar\'e duality, $H_2 (X, \op{Ad} \rho)$ is the dual of $H_1 (X, \op{Ad} \rho)$. So the Reidemeister torsion of $\op{Ad} \rho$ is a non vanishing element of $ \bigl( \det H_1 ( X, \op{Ad} \rho ) \bigr) ^{-2}$ well-defined up to sign. Consequently, the inverse of the square root of the torsion is a density of $H^1 (X, \op{Ad} \rho)$. Since $H^1 (X, \op{Ad} \rho)$ identifies with the tangent space of $\moir(X)$ at $\rho$, we define in this way a density $\mu_X$ on $\moir (X)$. 

For any $u \in \Co (G)$ and $[\rho] \in \moir ( \Si, u)$,  the tangent space $T_{[\rho]} \moir (\Si, u)$ is identified with the kernel of the morphism $ H^1( \Si, \op{Ad} \rho) \rightarrow H^1( \partial \Si , \op{Ad} \rho)$. The symplectic product of $T_{[\rho]} \moir (\Si, u)$ is induced by the intersection product of  $H^1( \Si, \op{Ad} \rho)$ with $H^1( \Si, \partial \Si, \op{Ad} \rho)$. We denote by $\mu_u$ the corresponding Liouville measure of $\moir (\Si, u)$. 

As a last definition, let $\Delta : \Co (G) \rightarrow \R$ be the function given by 
$$ \Delta ( u ) = \bigl| {\det} _{H_g} (\op{Ad}_g - \op{id} ) \bigr|^{1/2}  
$$   
where $g $ is any element in the conjugacy class $u$ and $H_g$ is the orthocomplement of $\ker ( \op{Ad}_g - \op{id})$. Equivalently, let $\mathfrak{t}$ be the Lie algebra of a maximal torus of $G$, $R \subset \mathfrak{t}^*$ be the corresponding set of real roots  and $R_+ \subset R$ be a set of positive roots. Then for any $X \in \mathfrak{t}$, 
$$ \Delta ( [ e^X ]) = \prod_{\al \in R_+ ; \; \al (X) \neq 0} 2| \sin ( \pi \al (X))| $$

\begin{theo} \label{theo:main-result}  
For any $(u,v) \in \mathcal{P}$, we have on $\moir (X , u, v)$
$$ \mu_X =  \Biggl(\prod_{i=1}^{n} \frac{ \Delta ( u_i^{r_i}) }{p_i^{(\dim G - \dim u_i )/2}} \Biggr) \;   R_{u,v} ^* \mu_u $$
where $R_{u,v}$ is the restriction map from $\moir (X,u,v)$ to $\moir ( \Si, u)$ and  for each $i$, $r_i$ is any inverse of $q_i$ modulo $p_i$, and $u_i^{r_i} \in \Co (G)$ is the conjugacy class containing the $g^{r_i}$ for $g \in u _i$.  
\end{theo}

Several definitions require an invariant scalar product on the Lie algebra of $G$: the symplectic structure of $\moir (\Si, u)$, the Poincar\'e duality between $H_1 (X, \op{Ad} \rho)$ and $H_2 ( X, \op{Ad} \rho)$ and the Reidemeister torsion of $\op{Ad} \rho$. Our implicit convention is to choose the same invariant scalar product each time.

During the proof, we will prove interesting intermediate results: 
\begin{itemize} 
\item[-] for any irreducible representation $\rho$ of $\pi_1(X)$ in $G$, the cohomology groups $H^1 (X, \op{Ad} \rho)$ and $H^2 ( X , \op{Ad} \rho)$ both identify naturally with the kernel of the restriction morphism $H^ 1( \Si , \op{Ad} \rho ) \rightarrow H^1 ( \partial \Si, \op{Ad} \rho)$. 
\item[-] by these identifications, the intersection product of $H^1 (X, \op{Ad} \rho)$ with $H^2 ( X , \op{Ad} \rho)$ is sent to the intersection product of $H^1 (\Si, \op{Ad} \rho ) $ with $H^1 (\Si , \partial \Si , \op{Ad} \rho)$. 
\item[-] the Reidemeister torsion $\tau ( \op{Ad} \rho, X)$ is equal to  $C ^{-2} \det \psi$ where $\psi : H_1 (X, \op{Ad} \rho) \rightarrow H_2 (X, \op{Ad} \rho)$ is the map induced by the previous identifications and $C$ is the factor appearing in Theorem \ref{theo:main-result}. 
\end{itemize}
This results are respectively proved in Sections \ref{sec:homology-groups}, \ref{sec:poincare-duality} and \ref{sec:torsion}. Theorem \ref{theo:main-result} is proved in Section \ref{sec:appl-moduli-spac} and Theorem \ref{theo:smooth} in Section \ref{sec:smooth-structure}. 

\subsection*{Related results in the litterature}  
Witten \cite{Wi_to} proved that for $S$ a closed oriented surface, the canonical density $\mu_S$ of $\moir (S)$  defined from Reidemeister torsion, is the Liouville measure of the natural symplectic structure of $\moir (S)$. He also extended this result to surfaces with boundary. We tried to deduce Theorem \ref{theo:main-result} from this by expressing the torsion $\tau (\op{Ad} \rho,X ) $ in terms of $\tau ( \op{Ad} \rho , \Si)$, without any success. Our actual proof does not use Witten's result.  

Witten also computed explicitely the volumes $\int_{\moir (\Si,u)  } \mu_u$, cf. \cite{Wi_to}, Formula 4.114. For $G = \op{SU} (2)$ and non central conjugacy classes $u_i$,  Park \cite{Pa} adapted the Witten's method to compute $\int_{\moir (X,u,v)} \mu_X$, $X$ being our Seifert manifold. Computing the volume of $\moir (X,u,v)$ with Theorem \ref{theo:main-result} and Witten's formula, we can extend Park's result to any compact connected Lie group $G$ with finite center and any conjugacy classes $u_i$.

McLellan \cite{McLe} proved a result similar to Theorem \ref{theo:main-result} for $G= \op{U}(1)$. To do this, he introduced a Sasakian structure on $X$ and used a computation of the corresponding analytic torsion \cite{RuSe}. We will explain in Section \ref{sec:abelian-case} how we can recover McLellan's result by adapting our method, providing an elementary proof. 


\section{The Seifert manifold $X$} \label{sec:seifert-manifold-x}

Let $g, n,  p_1, q_1, \ldots , p_n, q_n$  be integers such that  
\begin{gather} \label{eq:condition_gnpq}
 g \geqslant 0, \quad n \geqslant 1 \quad \text{ and }  \quad \forall i, \quad p_i, q_i \text{ are coprime and } p_i \geqslant 1. 
\end{gather}
To such a familly we associate the following manifold $X$. 
Let $\Si$ be a compact oriented surface with genus $g$ and $n$ boundary components denoted by $C_1$, \ldots, $C_n$. Let $D$ be a closed disc and for any $i$, let $\varphi_i :\partial D \times S^1 \rightarrow C_i \times S^1 $ be an orientation reversing diffeomorphism such that we have in $H_1 ( S^1 \times C_i)$,  
\begin{gather} \label{eq:coef_chirurgie} 
 [ \varphi_i ( \partial D)] = - p_i [C_i] + q_i [S^1].
\end{gather}
where $\partial D$ and $C_i$ are oriented as boundaries of $D$ and $\Si$ respectively. 
Then $X$ is obtained by gluing $n$ copies of $D\times S^1$ to $\Sigma \times S^1$ along its boundary through the maps $\varphi_i$, 
\begin{gather} \label{eq:Seifert_dec} 
X = (\Si \times S^1) \cup_{\varphi_1 \cup \ldots  \cup\varphi_n} ( D \times S^1) ^{\cup n } .
\end{gather}  
By construction $\Si \times S^1$ is a submanifold of $X$. In the sequel we often consider $\Si$ and $S^1$ as submanifolds of $X$ by identifying $\Si $ with $\Si \times \{ y\}$ and $S^1$ with $\{ x \} \times S^1$, where $x$ and $y$ are some fixed points of $\Si$ and $S^1$ respectively. 

The above definitions are all what we need for this article. Nevertheless, it is interesting to understand this in the context of Seifert manifolds. First, if $X$ is obtained as previously,  we can extend the $S^1$-action on $ \Si \times S^1$ to $X$, so that for any $i$, the action on the $i$-th copy of $D \times S^1$ is free if $p_i=1$ and otherwise it has one exceptional orbit with isotropy $\Z_{p_i}$. Conversely, consider any three dimensional closed connected oriented manifold  $Y$ equipped with an effective locally free action of $S^1$. Then choose $n\geqslant 1 $ orbits $O_1$, \ldots, $O_n$ of $Y$ including all the exceptional ones. Let $T_1$, $\ldots$, $T_n$ be disjoint saturated open tubular neighborhoods of the $O_1$, \ldots, $O_n$ respectively. Let $\Si$ be any cross-section of the action on $Y \setminus (T_1 \cup \ldots \cup T_n)$. For any $i$, set $C_i = (\partial \Si) \cap \overline{T}_i$ and define $p_i$ as the order of the isotropy group of $O_i$ and $q_i$ so that $[C_i] = q_i [O_i]$ in $H_1( \overline{T}_i ) $, where $C_i$ is oriented as the boundary of $\Si$ and $O_i$ by the $S^1$-action. Let $X$ be any manifold associated to  the data $\Si$, $(p_1, q_1)$, \ldots, $(p_n, q_n)$ as in (\ref{eq:Seifert_dec}). 
Then $Y$ is diffeomorphic to $X$, cf. \cite{JaNe}, Theorem 1.5 or the Section 1 of \cite{NeRa} for more details. We can even choose the diffeomorphism between $Y$ and $X$ so that it commutes with the $S^1$-action and fixes $\Si$. The collection 
 $$(g; (p_1, q_1), \ldots, ( p_n,  q_n))$$
is called the unnormalised Seifert invariant of $Y$.

\section{Character space of a Seifert manifold} \label{sec:char-space-seif}

\subsection*{Notations} 
Let $G$ be a Lie group. For any connected topological space $Y$,  we denote by $\mo (Y)$ the set of conjugacy classes of representations of $\pi_1 (Y)$ into $G$ \footnote{A representation of $\pi_1(Y)$ into $G$ is a group morphism from $\pi_1 (Y)$ to $G$. Two representations $\rho, \rho' $ are conjugate if there exists $g \in G$, such that $ \rho'(h)  = g \rho (h) g^{-1}$, $\forall h \in G$.}.
A representation $\rho : \pi_1 (Y) \rightarrow G$ is said to be irreducible if the centraliser of $\rho ( \pi_1 (Y))$ is reduced to the center of $G$. We denote by $\moir (Y)$ the subset of $\mo (Y)$ consisting of conjugacy classes of irreducible representations. 

If $Z$ is a subspace of $Y$, there is a natural morphism $j_*$ from $\pi_1(Z)$ to $\pi_1(Y)$ and consequently a natural map from $\mo (Y)$ to $\mo (Z)$, sending $[\varphi]$ into $[\varphi \circ j_* ]$. For any representation $\rho : \pi_1 (Y) \rightarrow G$, we call $\rho  \circ j_*$ the restriction of $\rho$ to $Y$. 

\subsection{A decomposition of $\moir( X)$}

From now on, $X$ is the Seifert manifold introduced in Section \ref{sec:seifert-manifold-x}. Recall that we view $S^1$ and $\Si$ as submanifolds of $X$.  
\begin{prop} \label{prop:irred}
Let $\rho$ be a representation of $\pi_1 ( X)$ into $G$. Then  $\rho$ is irreducible if and only if its restriction to $\Si$ is irreducible. Furthermore, if $\rho$ is irreducible, then $\rho (S^1)$ is central. Finally, for any $i$, $\rho ( C_i) ^{p_i}$ is conjugate to $\rho ( S^1) ^{ q_i} $.
\end{prop}

In the statement we slightly abused notation by applying $\rho$ to oriented circles of $X$. Since any loop $\ga$ of $X$ is homotopic to an element of $\pi_1 (X)$ unique up to conjugation,
the conjugacy class of $\rho ( \ga )$ is uniquely defined.  

\begin{proof} 
By Van Kampen theorem, the natural morphism $\pi_1 ( \Si \times S^1) \rightarrow \pi_1 (X)$ is onto. So $\rho : \pi_1 (X) \rightarrow G$ is irreducible if and only if its restriction to $\pi_1 ( \Si \times S^1)$ is irreducible. Since $\rho ( \pi_1 (\Si ) ) \subset \rho ( \pi_1 ( \Si \times S^1))$, if $\rho|_{\Si}$ is irreducible, then $\rho| _{\Si \times S^1}$ is irreducible. Conversely, assume that $\rho| _{\Si \times S^1}$ is irreducible. Since $\pi_1 ( \Si \times S^1) \simeq \pi_1 ( \Si ) \times \pi_1 ( S^1)$, $t = \pi_1 ( S^1)$ is in the centraliser of $\pi_1 ( \Si \times S^1)$, and consequently $\rho(t)$ is central.  This implies that $\rho ( \pi_1 ( \Si \times S^1))$ and $\rho ( \pi_1 ( \Si))$ have the same centraliser. So $\rho|_{\Si}$ is irreducible.  

By Equation (\ref{eq:coef_chirurgie}),  $\rho ( C_i) ^{p_i} $  and  $\rho ( S^1) ^{ q_i}$ are conjugate.
\end{proof} 

Let $Z(G)$ be the center of $G$ and $\Co (G)$ be the set of conjugacy classes. If $u \in \Co (G)$ and $p$ is an integer, $u^p \in \Co (G)$ is defined as the conjugacy class of $g^p$ where $g\in u$. Let $\mathcal{P}$ be the subset of $\Co (G)^n \times Z(G)$ consisting of the pairs $(u,v)$ such that for any $i$, $u^{p_i} = v^{q_i}$. Then by the last part of Proposition \ref{prop:irred}, 
\begin{gather} \label{eq:dec_moirX} 
 \moir ( X) = \bigcup_{(u,v) \in \mathcal{P}} \moir ( X, u, v )
\end{gather}
where $\moir ( X, u,v)$ consists of the $[\rho] \in \moir (X)$ such that $\rho (S^1)\in  v$ and $\rho (C_i) \in u_i$ for any $i$. Denote by $R_{u,v}$ the restriction map
\begin{gather} \label{eq:Ruv}
 R_{u,v} : \moir ( X, u,v) \rightarrow \moir (\Si, u), \qquad [\rho] \rightarrow [\rho|_{\Si}]
\end{gather}
where $\moir ( \Si , u)$ is the subset of $\moir (\Si)$ consisting of the classes $[\rho]$ such that  for any $i$, $\rho (C_i) \in u_i$.

\begin{prop} For any $(u, v ) \in \mathcal{P}$, the map $R_{u,v}$ is a bijection.
\end{prop}  

\begin{proof} It is a consequence of the fact that $\pi_1 ( \Si \times S^1) = \pi_1 ( \Si ) \times \pi_1 (S^1)$ and that the kernel of the surjective map $\pi_1 ( \Si \times S^1) \rightarrow \pi_1 (X)$ is the normal subgroup normally generated by the $\varphi_i ( \partial D)$'s.
\end{proof}

\subsection{Topology and manifold structure}  \label{sec:smooth-structure}

From now on, assume that $G$ is compact and has a finite center. 
As explained in appendix \ref{sec:representation-space}, for any compact connected manifold $Y$, $\mo (Y)$ has a natural Hausdorff topology and $\moir (Y)$ is an open subset. 

\begin{lem} 
The set $\mathcal{P}$ is finite. For any $(u,v) \in \mathcal{P}$, $\moir ( X, u,v)$ is an open subset of $\moir (X)$. 
\end{lem}

\begin{proof} By identifying $\Co (G)$ with the quotient of a maximal torus by the Weyl group, we easily see that for any $v \in \Co (G)$ and $p\in \Z$, the equation $u^p = v$ has only a finite number of solutions.    This implies that $\mathcal{P}$ is finite. We deduce that the $\moir (X,u,v)$'s are open by applying the following fact: for any compact connected manifold $Y$, for any $x \in \pi_1 (Y)$, the map from $\mo (Y)$ to $\Co (G) $ sending $[\rho]$ into $[\rho (x)]$ is continuous.   
\end{proof}

By Appendix \ref{sec:representation-space}, $\moir (X)$ has a natural open subset $\mo ^{\op{s}, 0} (X)$ which is a manifold. 
Furthermore, it is known that the spaces $\moir (\Si, u)$ are smooth manifolds. 

\begin{prop} \label{prop:smooth} 
We have $\moir (X) = \mo ^{\op{s}, 0} (X)$. Furthermore, for any $(u,v) \in \mathcal{P}$,  $R_{u,v}$ is a diffeomorphism from $\moir (X, u, v)$ to $\moir (\Si, u)$. 
\end{prop} 

 It is possible that the various $\mo ^0 (X, u,v)$ have different dimensions. Actually, 
$$ \dim  \mo ^0 ( \Si, u) = 2(g -1) \dim G + \sum_{i=1}^{n} \dim u_i .$$

\begin{proof} 
Let $u \in \Co (G)^n$ and consider the set $M_u$ of $(a,b,c) \in (G^{2g+n})^0 $ satisfying the relations
$$ [a_1, b_1] \ldots [a_g, b_g] c_1\ldots c_n = \op{id}, \qquad c_i \in u_i, \quad \forall i .$$ 
Here we used the same notation $(G^{2g+n})^0$ as in Appendix \ref{sec:representation-space}. It is known that $M_u$ is a smooth submanifold of $G^{2g +n}$. 

Choose a standard set of generators $(x,y,z)$ of $\pi_1 (\Si)$ and let $t\in \pi_1 (X)$ be isotopic to $S^1$.  The map $\pi_1 ( \Si \times S^1) \rightarrow \pi_1 ( X)$ being onto,  $(x,y,z,t)$ is a set of generators of $\pi_1 (X)$. Through these generators, $\repir (\pi_1(X))$ gets identified with a subset $A$ of $G^{2g+n+1}$ as explained in Appendix \ref{sec:representation-space}. By the decomposition (\ref{eq:dec_moirX}), $A$ is the union of the $M_u \times \{v\}$ where $(u,v)$ runs over $\mathcal{P}$. $\mathcal{P}$ being finite, $A$ is a submanifold of $G^{2g+n+1}$, which shows that $\repir (\pi_1(X)) = \rep ^{\op{s} , 0 } ( \pi_1(X))$ in the notation of Appendix \ref{sec:representation-space} and consequently that $\mo ^{\op{s}, 0} (X) = \moir (X)$. For any $(u,v) \in \mathcal{P}$, the projection $M_u \times \{ v \} \rightarrow M_u$ being a diffeomorphism, we conclude that $R_{u,v}$ is a diffeomorphism. 
\end{proof}

Consider again a compact connected manifold and a representation $\rho$ of $\pi_1 (Y)$ in $G$. Composing $\rho$ with the adjoint representation, the Lie algebra $\mathfrak{g}$ becomes a $\pi_1(Y)$-module. Denote by $H ^{\bullet} ( \pi_1(Y) , \op{Ad} \rho)$ the group cohomology with coefficient in $\mathfrak{g}$. Alternatively, we may consider the flat vector bundle $\op{Ad} \rho \rightarrow Y$ associated to $\rho$ via the adjoint representation. Let $H ^{\bullet} ( Y, \op{Ad} \rho)$ be the cohomology of $Y$ with local coefficient. Then for $j=0$ or $1$, $H^ j (\pi_1(Y) , \op{Ad} \rho) \simeq H^j ( Y, \op{Ad} \rho)$.

\begin{lem} \label{lem:tangent}
For any irreducible representation $\rho$ of $\pi_1(X)$, we have a natural identification between $H^1 ( X, \op{Ad} \rho ) $ and $ T_{[\rho]} \moir (X)$. 
\end{lem} 
\begin{proof} 
By appendix \ref{sec:representation-space}, $T_{[\rho]} \moir (X)$ is naturally identified with a subspace of $H^1 ( X, \op{Ad} \rho)$. 
Similarly, it is known that $T_{[\rho]} \mo ^0 ( \Si, u)$ gets identified to the kernel of the morphism $H^1 ( \Si, \op{Ad} \rho) \rightarrow H^1 ( \partial \Si, \op{Ad} \rho)$. 
Furthermore, we easily see that the tangent linear map to $R_{u,v}$ is the restriction of the morphism $ H^1 ( X, \op{Ad} \rho) \rightarrow H^1(\Si, \op{Ad} \rho)$. As we will see in Theorem \ref{theo:homologie_Seifert},  the following sequence is exact
$$ 0 \rightarrow H^1 ( X, \op{Ad} \rho) \rightarrow H^1 ( \Si, \op{Ad} \rho) \rightarrow H^1 ( \partial \Si, \op{Ad} \rho) \rightarrow 0 .$$
 This implies that $H^1 ( X, \op{Ad} \rho )=  T_{[\rho]} \moir (X)$.
\end{proof}

\section{The homology groups $H_1 ( X, \op{Ad} \rho)$ and $H_1 ( \Si, \op{Ad} \rho)$} \label{sec:homology-groups}

As in the previous section, for any compact connected topological space $Y$ and representation $\rho : \pi_1 ( Y) \rightarrow G$, we consider the flat vector bundle $\op{Ad} \rho \rightarrow Y$. We are interested in corresponding homology groups $H_{\bullet} ( Y, \op{Ad} \rho)$ for $Y= X$ or $\Si$. As a first remark, if $\rho$ is irreducible, then by Appendix \ref{sec:representation-space}, $H^0 ( Y, \op{Ad} \rho) = H^0 ( \pi_1(Y), \op{Ad} \rho)= 0$ because the center of $G$ is finite. By duality, $H_0 ( Y, \op{Ad} \rho) =0$.



Consider the surface $\Si$ and an irreducible representation $\rho : \pi_1 ( \Si ) \rightarrow  G$. For any boundary component $C_i$, choose a base point $x_i \in C_i$ and let $V_i = \ker ( \op{hol}_i - \op{id})$ where $\op{hol}_i :  \op{Ad} \rho |_{x_i} \rightarrow  \op{Ad} \rho |_{x_i}$ is the holonomy of $C_i$ in $ \op{Ad} \rho $. We have two isomorphisms
$$ H_0 ( C_i ,   \op{Ad} \rho ) \simeq V_i, \quad H_1 ( C_i ,  \op{Ad} \rho) \simeq V_i. $$
sending $u\in V_i$ into $ [x_i ] \otimes u$ and $[C_i ] \otimes u$ respectively. 

\begin{lem} \label{lem:Sigma}
We have $H_0 ( \Si, \op{Ad} \rho ) = H_2 ( \Si , \op{Ad} \rho ) =0$. Furthermore
the natural map $f: H_1 ( \partial \Si,  \op{Ad} \rho ) \rightarrow H_1 ( \Si , \op{Ad} \rho )$ is injective.
\end{lem} 

\begin{proof} 
$\Si $ being connected with a non empty boundary, $H_0 ( \Si, \partial \Si ,  \op{Ad} \rho ) = 0$, so by Poincar{\'e} duality, $H_2 ( \Si ,  \op{Ad} \rho) = 0 $. Since $\rho $ is irreducible, $H_0 ( \Si , \op{Ad} \rho ) =0$ and by Poincar{\'e} duality, $H_2 ( \Si, \partial \Si ,  \op{Ad} \rho ) =0$. Writing the long exact sequence associated to the pair $(\Si, \partial \Si)$, we deduce that $f$ is  one-to-one.
\end{proof} 

Consider now the Seifert manifold $X$ and an irreducible representation $\rho  : \pi_1 (X) \rightarrow G $. Since $\Si$ is a submanifold of $X$, we have a natural morphism  
$$ g: H_1 ( \Si, \op{Ad} \rho ) \rightarrow H_1 ( X, \op{Ad} \rho)$$
By lemma \ref{prop:irred}, the restriction of $\rho$ to $S^1$ is central. So 
the restriction of the bundle $\op{Ad} \rho$ to $\Si \times S^1$ is isomorphic to $\op{Ad} \rho |_X \boxtimes \R_{S^1}$ \footnote{If $E \rightarrow B$ and $E' \rightarrow B'$ are two vector bundles, we denote by $E \boxtimes E'$ the vector bundle $(\pi ^* E) \otimes ((\pi')^* E' )$ where $ \pi$ and $\pi'$ are the projection from $B \times B'$ onto $B$ and $B'$.}. Here we denote by $\R_{S^1}$ the trivial vector bundle over $S^1$ with fiber $\R$. This allows to define a second application 
$$ h : H_1 ( \Si, \op{Ad} \rho) \rightarrow H_2 ( X, \op{Ad} \rho)$$
which sends $\al \in H_1 ( \Si, \op{Ad} \rho)$ into the image of $ \al \boxtimes [S^1] \in H_2 ( \Sigma \times S^1, \op{Ad} \rho )$ by the natural morphism $H_2  ( \Sigma \times S^1, \op{Ad} \rho ) \rightarrow  H_2 ( X, \op{Ad} \rho)$.
 
\begin{theo} \label{theo:homologie_Seifert}
We have $H_0 ( X, \op{Ad} \rho) = H_3 ( X, \op{Ad} \rho) = 0$. Furthermore the following sequences are exact:
$$ 0 \rightarrow H_1 ( \partial \Si , \op{Ad} \rho) \xrightarrow{f} H_1 ( \Si, \op{Ad} \rho ) \xrightarrow{g} H_1 ( X, \op{Ad} \rho) \rightarrow 0, $$
$$ 0 \rightarrow H_1 ( \partial \Si , \op{Ad} \rho ) \xrightarrow{f} H_1 ( \Si,  \op{Ad} \rho) \xrightarrow{h} H_2 ( X, \op{Ad} \rho ) \rightarrow 0 .$$
\end{theo}

\begin{proof} Since $\rho$ is irreducible, $H_ 0 ( X, F) = 0 $. By Poincar{\'e} duality, $H_3 ( X, F) = 0$. 
To prove that the sequences are exact, we will consider the Mayer-Vietoris long exact sequence associated to the decomposition (\ref{eq:Seifert_dec}) of $X$.

Since the restriction of $\op{Ad} \rho$ to $\Si \times S^1$ is isomorphic to $\op{Ad} \rho |_{\Si}  \boxtimes \R_{S^1}$, we can compute by applying the K{\"u}nneth theorem to the maps
\begin{gather} \label{eq:map}
H_j ( \partial \Si \times S^1, \op{Ad} \rho ) \rightarrow H_j ( \Si \times S^1, \op{Ad} \rho ), \qquad j =3,2,1,0. 
\end{gather} 
We have that $H_j ( S^1, \R) = \R$ for $j=0,1$ and by Lemma \ref{lem:Sigma}, $H_j (\Si,  \op{Ad} \rho) =0$  for $j=0,2$.  We deduce that 
$$H_3 ( \Si \times S^1, \op{Ad} \rho ) = H_0 ( \Si \times S^1, \op{Ad} \rho) =0$$
and $H_0 (\partial \Si \times S^1, \op{Ad} \rho ) \simeq H_0 ( \partial \Si, \op{Ad} \rho) $, which determines (\ref{eq:map}) for $j=0$ and $3$. For $j=2$,  the map (\ref{eq:map}) identifies with the map 
$f: H_1 (\partial \Si ,  \op{Ad} \rho ) \rightarrow H_1 ( \Si,  \op{Ad} \rho  )$ and for $j =1$ with
$$ f \oplus 0: H_1 ( \partial \Si ,  \op{Ad} \rho ) \oplus H_0 ( \partial \Si,  \op{Ad} \rho  ) \rightarrow H_1 ( \Si ,  \op{Ad} \rho  )  ,$$
because $H_0( \Si , \op{Ad} \rho) =0$. 
Applying again the K{\"u}nneth theorem, the maps $H_j ( \Si \times S^1,   \op{Ad} \rho  ) \rightarrow H_j ( X,   \op{Ad} \rho  )$ identify with $g$ and $h$ for $j=1$ and $2$ respectively.  

It remains to compute the maps $H_j ( C_i \times S^1,  \op{Ad} \rho ) \rightarrow H_j ( D \times S^1, \tilde{\varphi}_i ^*  \op{Ad} \rho  )$. Here we denote by $\tilde{\varphi}_i$ the embedding of $D\times S^1$ into $X$ extending $\varphi_i$. 
Since $D$ is contractible, $H_j ( D \times S^1, \tilde{\varphi}_i ^*  \op{Ad} \rho  ) = 0$ for $j =2,3$. Let us determine the holonomy of $S^1$ in the bundle $\tilde{\varphi}_i ^*  \op{Ad} \rho  \rightarrow D\times S^1$. It is equal to the holonomy of $\varphi _i ( S^1)$ in $  \op{Ad} \rho  \rightarrow C_i \times S^1$.  For any loop $\ga$ of $ C_i \times S^1 $ based at $(x_i, 0)$, we denote by $\op{hol}_\ga :  \op{Ad} \rho |_{x_i} \rightarrow  \op{Ad} \rho |_{x_i} $ the holonomy of $\ga$ in $ \op{Ad} \rho  \rightarrow C_i \times S^1$. Since $D$ is contractible, $\op{hol } _{ \varphi_i (\partial D) }$ is trivial, so that   
\begin{xalignat*}{3} 
 \ker  ( \op{hol} _{\varphi_i ( S^1)} - \op{ id} ) =  & \ker  ( \op{hol} _{\varphi_i (\partial D)} - \op{ id} )  \cap  \ker  ( \op{hol} _{\varphi_i ( S^1)} - \op{ id} )  \\
= & \ker  ( \op{hol} _{C_i } - \op{ id} )  \cap  \ker  ( \op{hol} _{S^1} - \op{ id} ) \\
= & \ker  ( \op{hol} _{ C_i} - \op{ id} ) \\ = & V_i 
\end{xalignat*}
where we have used first that $\varphi_i$ is a diffeomorphism and second that $\op{hol}_{S^1}$ is trivial. 
We deduce that 
$$H_j ( D\times S^1,\tilde{\varphi}_i ^*  \op{Ad} \rho  ) \simeq V_i$$ for $j =0$ or $1$. As above, let us identify $H_1 ( C_i \times S^1,  \op{Ad} \rho  ) $ with $ H_1 ( C_i , \op{Ad} \rho ) \oplus H_0 (C_i ,  \op{Ad} \rho  ) = V_i \oplus V_i $. Then by Equation (\ref{eq:coef_chirurgie}), the map $H_1 ( C_i \times S^1,  \op{Ad} \rho ) \rightarrow H_1 (  D\times S^1,  \tilde{\varphi}_i ^*  \op{Ad} \rho )$  corresponds to 
$$ V_i \oplus V_ i \rightarrow V_i , \qquad ( u,v ) \rightarrow q_i u + p_i v .$$
Putting everything together and setting $V = \bigoplus V_i$, we obtain the following long exact sequence
\begin{gather*}
 0 \rightarrow V \xrightarrow{f} H_1 ( \Si ,  \op{Ad} \rho ) \xrightarrow{h} H_2 ( X,  \op{Ad} \rho ) \rightarrow V \oplus V \xrightarrow{\tilde{f} }  \\ H_1 ( \Si ,  \op{Ad} \rho  ) \oplus V  \xrightarrow{ [g, \tilde{g} ]} H_1 ( X,  \op{Ad} \rho ) \rightarrow V \xrightarrow{ \op{ id}} V \rightarrow 0 
\end{gather*} 
where $\tilde{g}: V \rightarrow H_1 ( X ,  \op{Ad} \rho )$ is unknown and $\tilde{f}$ is the map $\begin{pmatrix} f & 0 \\ q & p \end{pmatrix}$ with $q, p : V \rightarrow V$  the maps whose restriction to $V_i$ are the multiplications by $q_i$, $p_i$ respectively. 

We recover the fact that $f$ is injective. Since $f$ is injective and the $p_i$ don't vanish, $\tilde{f}$ is injective too. Furthermore the identity of $V$ is certainly injective. So the Mayer-Vietoris long exact sequences breaks into three exact sequences:
\begin{gather} \label{eq:2} 
  0 \rightarrow V \xrightarrow{f} H_1 ( \Si ,  \op{Ad} \rho ) \xrightarrow{h} H_2 ( X,   \op{Ad} \rho ) \rightarrow  0\\ 
\label{eq:3} 0 \rightarrow V \oplus V \xrightarrow{\tilde{f} }   H_1 ( \Si ,  \op{Ad} \rho ) \oplus V  \xrightarrow{ [g, \tilde{g} ]} H_1 ( X,   \op{Ad} \rho ) \rightarrow  0 \\
\label{eq:7} 0 \rightarrow V \xrightarrow{\op{id}} V \rightarrow 0
\end{gather}
(\ref{eq:2}) is the second exact sequence in the statement of the theorem. 
Finally, it is an easy exercise to deduce from the exact sequence (\ref{eq:3}) that the first sequence in the statement of the theorem is exact. 
\end{proof} 

\section{Poincar\'e duality on $X$ and $\Si$}  \label{sec:poincare-duality}

Choose an invariant scalar product on the Lie algebra of $G$. For any topological space $Y$ and representation $\rho$ of $\pi_1 (Y) $ in $G$, the flat vector bundle $\op{Ad} \rho$ inherits a flat metric. This allows to define a cup product $H^k ( Y, Z, \op{Ad} \rho) \times H^\ell ( Y, Z,\op{Ad} \rho ) \rightarrow H^{k + \ell} ( Y,Z,  \R )$ for any closed subspace $Z$ of $Y$. We will us these products for $X$ and $\Sigma$. 

Consider an irreducible representation $\rho$ of $\pi_1(\Si)$ in $G$. We have a bilinear map 
\begin{gather} \label{eq:pairing}
H^1 ( \Si , \partial \Si ,  \op{Ad} \rho ) \times H^{1} ( \Si ,   \op{Ad} \rho  ) \rightarrow \R, \quad (\al, \be) \rightarrow \al \cdot \be
\end{gather}
sending $( \al, \be)$ to the evaluation of the cup product $\al \cup \be$ on the fundamental class of ($\Si$, $\partial \Si$). 
Consider the following portion of the long exact sequence associated to the pair $( \Si, \partial \Si )$
$$ ... \rightarrow H^{1} ( \Si , \partial \Si ,  \op{Ad} \rho ) \xrightarrow{\pi} H^1 ( \Si ,  \op{Ad} \rho ) \xrightarrow{f^*} H^1 ( \partial \Si ,  \op{Ad} \rho ) \rightarrow ... $$
and introduce the space $K := \ker  f^* =  \op{Im}  \pi \subset H^1 ( \Si ,  \op{Ad} \rho )$.
For any $\al$, $\be \in K$, we set
\begin{gather} \label{eq:defOm}
 \Om ( \al  , \be ) :=  \tilde{\al} \cdot \be 
\end{gather}
where $\tilde{\al} $ is any element of $H^1( \Si, \partial \Si,  \op{Ad} \rho )$ such that $\pi ( \tilde {\al} ) = \al$.

\begin{lem} \label{lem:symp_Structure}
The bilinear map $\Om$ is well-defined, antisymmetric and non degenerate.
\end{lem} 

So $(K, \Om)$ is a symplectic vector space. 
\begin{proof}
For any $ \tilde \al$, $ \tilde \be \in H^{1} ( \Si , \partial \Si , \op{Ad} \rho )$, 
$\tilde \al \cdot \pi (\tilde \be) +  \tilde \be \cdot \pi ( \tilde \al ) = 0$. Assuming that $ \pi ( \tilde \al ) = \al$ and $\pi ( \tilde \be ) = \beta$, we get that
$$ \Om ( \al , \be ) =  \tilde{\al} \cdot \be = - \tilde{\be} \cdot \al ,$$
 which proves that $\Om ( \al, \be)$ does not depend on the choice of $\tilde \al$ and that $\Om ( \al , \be ) = - \Om ( \be , \al)$. By Poincar{\'e} duality, the pairing (\ref{eq:pairing}) is non degenerate, so the same holds for $\Om$. 
\end{proof}

Consider now an  irreducible representation $\rho$ of $\pi_1(X)$ in $G$. By Poincar\'e duality, we have a nondegenerate pairing
\begin{gather} \label{eq:poincareX}
 H^1 ( X, \op{Ad} \rho) \times H^2 ( X,  \op{Ad} \rho) \rightarrow \R 
\end{gather}
sending $(\al, \be)$ to the evaluation of $\al \cup \be \in H^3(X)$ on the fundamental class. 
By Theorem \ref{theo:homologie_Seifert}, 
the maps $g^*$ and $h^*$ induce isomorphisms from $H^1 ( X,  \op{Ad} \rho)$ and $H^{2}(X,  \op{Ad} \rho)$ to $K$. 

\begin{theo} \label{theo:symplectomorphism}
For any $\al \in H^1 ( X, \op{Ad} \rho)$ and $\be \in H^{2}(X, \op{Ad} \rho)$, we have
$$  \al \cdot \be = \Om ( g^* \al , h^* \be).
$$
\end{theo}
\begin{proof}
We will use de Rham cohomology. First let us prove that any element in $H^2 ( X, \op{Ad} \rho)$ has a representative $\be \in \Om^{2} ( X, \op{Ad} \rho)$ which vanishes identically on a neighborhood of $Z = \tilde{\varphi}_1 ( D \times S^1) \cup \ldots \cup 
\tilde{\varphi}_n ( D \times S^1)$ and such that 
\begin{gather} \label{eq:10} 
 \be = p_{\Si}^* \tilde{\be} \wedge p_{S^1}^* \tau + d \ga \quad \text{ on } \Si \times S^1
\end{gather}
where $p_{\Si}$ and $p_{S^1}$ are the projection from $\Si \times S^1$ onto $\Si$ and $S^1$ respectively, $\tilde \be \in \Om^1 ( \Si, \op{Ad} \rho)$ is closed, $\tau \in \Om^{1} (S^1)$ satisfies $\int_{S^1} \tau = 1$ and $\ga$ belongs to $\Om^1 ( \Si \times S^1, \op{Ad} \rho)$.
 
To check that, let us start with any representative $\be \in \Om^2 (X, \op{Ad} \rho)$. Since $H^2 (Z, \op{Ad} \rho) =0$, we have $\be = d \mu $ on $Z$. We can even assume that this holds on a neighborhood $U$ of $Z$. Let $\varphi \in \Ci ( X)$ with support contained in $U$ and identically equal to 1 on $Z$. Replacing $\be $ with $\be - d ( \varphi \mu)$, we have that $\be \equiv 0$ on a neighborhood of $Z$. By K{\"u}nneth theorem, $H^2 ( \Si \times S^1, \op{Ad} \rho) = H^1 ( \Si , \op{Ad} \rho) \otimes H^1 ( S^1, \R)$, which implies that $\be$ has the form (\ref{eq:10}) on $\Si \times S^1$.

Let us prove that any element in $H^1 ( X, \op{Ad} \rho)$ has a representative $\al \in \Om^1 ( X, \op{Ad} \rho)$ such that 
\begin{gather}  \label{eq:11}
 \al = p_{\Si } ^* \tilde \al  \qquad \text{on } \Si \times S^1
\end{gather}
where $\tilde \al \in \Om^1 ( \Si, \op{Ad} \rho)$ is closed and vanishes identically on $\partial \Si$. 

To check that, we start with any representative $\al \in \Om^1 ( X, \op{Ad} \rho)$. By K{\"u}nneth theorem, $H^{1} ( \Si \times S^1, \op{Ad} \rho) = H^1 ( \Si, \op{Ad} \rho) $ so that we have on $\Si \times S^1$ the equality $ \al = p_{\Si} ^* \tilde \al + d\ga $ with $\tilde \al \in \Om ^1 ( \Si, \op{Ad} \rho)$ closed and $\ga \in \Om^0 ( \Si \times S^1, \op{Ad} \rho)$. Observe that $[\tilde \al ] = g^* [\al]$, so by Theorem \ref{theo:homologie_Seifert}, $f^* [\tilde \al ] =0$. Thus adding to $\tilde \al$ an exact form (which modifies $\ga$), the restriction of $\tilde \al$ to $\partial \Si$ vanishes.  Finally, extending $\ga$ to $X$, and replacing $\al $ by $\al - d \ga$, we obtain Equation (\ref{eq:11}).

Now consider $\al$ and $\be $ as above. Then 
$$ [\al ] \cdot [\be] = \int_X \al \wedge \be = \int_{\Si \times S^1} \al \wedge \be$$
because $\be$ vanishes identically on a neighborhood of $Z$. To evaluate this last integral, we replace $\al$ and $\be$ by their expressions (\ref{eq:11}), (\ref{eq:10}). By Stokes' theorem,
$$ \int_{\Si \times S^1} p_{\Si}^* \tilde \al \wedge d \ga = \int_{\partial \Si \times S^1} p_{\Si}^* \tilde \al \wedge \ga  = 0$$
because $\tilde \al$ vanishes identically on $\partial \Si$. By Fubini theorem and because $\int_{S^1} \tau = 1$,
$$ \int_{\Si \times S^1}  p_{\Si}^* \tilde \al \wedge p_{\Si}^* \tilde{\be} \wedge p_{S^1}^* \tau = \int_\Si \tilde \al \wedge \tilde \be$$
Since $\tilde \al$ vanishes on $\partial \Si$, it is the representative of a class in $H^{1} ( \Si , \partial \Si , \op{Ad} \rho)$. So this last integral is equal to $\Om ( [\tilde \al ], [\tilde \be])$. Furthermore $[\tilde \al ] = g^* [\al]$ and $[\tilde \be] = h ^* [\be]$, which concludes the proof.
\end{proof} 

\section{Torsion of $X$}  \label{sec:torsion}

Let $\rho$ be an  irreducible representation $\rho$ of $\pi_1(X)$ in $G$. Since $H_0 ( X, \op{Ad} \rho ) = H_3 ( X , \op{Ad} \rho ) =0$, the torsion of the flat euclidean vector bundle $\op{Ad} \rho$ is a non vanishing vector of the line  
$$ \op{det} H_\bullet ( X, \op{Ad} \rho) \simeq \bigl( \op{det} ( H_1 ( X, \op{Ad} \rho )) \bigr)^{-1} \otimes \op{det} ( H_2 ( X, \op{Ad} \rho)) $$
well-defined up to sign. In the appendix \ref{sec:reidemeister-torsion}, we recall its definition and the properties we will need to compute it. By Theorem \ref{theo:homologie_Seifert}, we have an isomorphism 
$$\psi : H_1 ( X, \op{Ad} \rho ) \rightarrow H_2 ( X, \op{Ad} \rho ) $$ 
sending $g( \be)$ into $h( \be)$ for any $\be \in H_1 ( \Si , \op{Ad} \rho)$. The determinant of $\psi$ belongs to $\det  H_\bullet ( X, \op{Ad} \rho)$.

Let $\Delta : \Co (G) \rightarrow \R$ be the function given by 
$$ \Delta ( u ) = \bigl| {\det} _{H_g} (\op{Ad}_g - \op{id} ) \bigr|^{1/2}  
$$   
where $g $ is any element in the conjugacy class $u$ and $H_g$ is the orthocomplement of $\ker ( \op{Ad}_g - \op{id})$.

\begin{theo} \label{theo:torsion}
For any irreducible representation $\rho$ of $\pi_1 (X)$ in $G$, the torsion of $\op{Ad} \rho \rightarrow X$ is given by  
\begin{gather} \label{eq:torsion} 
\tau ( \op{Ad} \rho ) =  \prod_{i=1}^{n} \frac{p_i^{\dim V_i}}{ \Delta ^2 ( \rho (C_i)^{r_i} ) }
\det \psi 
\end{gather}
where $r_i$ is any inverse of $q_i$ modulo $p_i$ and $V_i = \ker  (\op{Ad}_{\rho (C_i)} - \op{id} ) $. 
\end{theo}

Let us make a few remark on the left hand side of (\ref{eq:torsion}).
\begin{enumerate}
\item It follows from the relation (\ref{eq:coef_chirurgie}) and the fact that $\rho (S^1)$ is central by Lemma \ref{prop:irred}, that $(\op{Ad}_{\rho (C_i)} ) ^{p_i}$ is the identity. So the right hand side of (\ref{eq:torsion}) does not depend on the choice of $r_i$. 
\item $V_i$ is the Lie algebra of the centralizer of $\rho (C_i)$ in $G$. So the dimension of $V_i$ is equal to $\dim G - \dim u_i$ where $u_i$ is the conjugacy class of $\rho (C_i)$. 
\end{enumerate}

\begin{proof} 
By the proof of Theorem \ref{theo:homologie_Seifert}, the Mayer-Vietoris long exact sequence breaks into three short exact sequences: (\ref{eq:2}), (\ref{eq:3}) and (\ref{eq:7}). 
Choose $\al \in \det V$ and $\be \in \bigwedge^{\dim H_1 ( \Si, \op{Ad} \rho) - \dim V}  H_1 ( \Si , \op{Ad} \rho)$ such that $ f( \al )   \wedge \be \in \det H_1 ( \Si, \op{Ad} \rho)$ does not vanish.
By (\ref{eq:2}), we have an isomorphism 
\begin{gather} \label{eq:4}
 \R \simeq  \det V \otimes \bigl( \det H_1 ( \Si, \op{Ad} \rho)\bigr)^{-1} \otimes  \det H_2 ( X, \op{Ad} \rho)
\end{gather} 
sending $1$ into $\al  \otimes \bigl( f(\al) \wedge \be \bigr)^{-1} \otimes h( \be)$. By (\ref{eq:3}), we have an isomorphism 
\begin{gather} \label{eq:5}
 \R \simeq ( \det V)^{-2} \otimes \bigl( \det H_1 ( \Si, \op{Ad} \rho) \otimes \det V \bigr)  \otimes  \bigl( \det H_1 (X, \op{Ad} \rho) \bigr)^{-1}
\end{gather} 
sending $1$ into $ \al ^{-2} \otimes \bigl( (f( \al) \wedge \be) \otimes ( \det p ) \al \bigr)  \otimes g( \be) ^{-1}$ where $p$ is the map introduced in the proof of Theorem \ref{theo:homologie_Seifert}. We easily compute that:
$$ \det p = \prod_{i=1}^{n} p_i^{ \dim V_i} $$
By (\ref{eq:7}), we have an isomorphism 
\begin{gather} \label{eq:8}
\R \simeq  \det V  \otimes \bigl( \det V \bigr)^{-1}
\end{gather} 
sending $1$ into $\al \otimes \al ^{-1}$.
Taking the tensor product of (\ref{eq:4}), (\ref{eq:5}) and (\ref{eq:8}), we get the isomorphism associated to the Mayer-Vietoris long exact sequence:
\begin{gather} \label{eq:9}
 \R \simeq  \bigl( \det  H_1 ( X, \op{Ad} \rho) \bigr)^{-1}  \otimes \det H_2 ( X, \op{Ad} \rho).
\end{gather}
It  sends $1$ into $ \bigl( \det p \bigr)  h ( \be) / g( \be) = \bigl( \det p \bigr)  \bigl( \det \psi \bigr) $.

Let us compute the torsion of the restrictions of $\op{Ad} \rho$ to $ C_i \times S^1$, $\Sigma \times S^1$ and $\tilde{\varphi}_i (D \times S^1)$ respectively. We will use the identifications made previously for the various cohomology groups. First, the torsion of $ \op{Ad} \rho \rightarrow C_i \times S^1$ is 
$$ 1 \in \R \simeq \det V_i \otimes \bigl( \det V_i \bigr)^{-1} \otimes \op{det} V_i \otimes \bigl( \det V_i \bigr)^{-1} .$$
Indeed, the bundle $\op{Ad} \rho|_{C_i \times S^1}$ is isomorphic to $\op{Ad} \rho|_{C_i} \boxtimes \R_{S^1}$. Furthermore, $\chi( C_i) = \chi (S^1) =0$. By property \ref{item:prod} of the appendix \ref{sec:reidemeister-torsion}, this implies that $\tau (  \op{Ad} \rho, {C_i \times S^1} ) = 1$. 

Second the torsion of $\op{Ad} \rho \rightarrow \Si \times S^1$ is 
$$ 1 \in \R \simeq \det H_1 ( \Si, \op{Ad} \rho) \otimes \bigl( \det H_1 ( \Si, \op{Ad} \rho) \bigr)^{-1} $$ 
Indeed, the bundle $\op{Ad} \rho|_{\Si \times S^1}$ is isomorphic to $\op{Ad} \rho|_{\Si} \boxtimes \R_{S^1}$. Since $\chi ( S^1) =0$, we deduce from properties \ref{item:prod} and \ref{item:cercle}  of the appendix \ref{sec:reidemeister-torsion} that $\tau ( \op{Ad} \rho, {\Si \times S^1} ) = \tau ( \R_{S^1}) ^{\chi (\op{Ad} \rho|_\Si)} = 1$.  

Third the torsion of $\tilde{\varphi}_i ^* \op{Ad} \rho \rightarrow  D \times S^1$ belongs to $\R \simeq \det V_i \otimes \bigl( \det V_i \bigr)^{-1}$. Since $\varphi_i$ is a diffeomorphism from $\partial D \times S^1$ to $C_i  \times S^1$ reversing the orientation and satisfying (\ref{eq:coef_chirurgie}), we have the following relation in $H_1 (C_i \times S^1)$
$$ \varphi_i ( [S^1] ) = r_i [C_i] + s_i [S^1] $$
where $r_i$, $s_i$ are such that such that $ p_i s_i + q_i r_i = 1$. Since $\rho ( S^1)$ is central, $\op{Ad}_{\rho (S^1)} $ is the identity, so $$\op{Ad }_{\rho ( \varphi_i ( S^1 ))} = \op{Ad}_{\rho ( C_i)} ^{r_i}.$$ 
By Property \ref{item:cercle}  of the appendix \ref{sec:reidemeister-torsion}, we conclude that the torsion of $\tilde{\varphi}_i ^* \op{Ad} \rho$ is equal to the square of $\Delta ( \rho (C_i) ^{r_i} ) $.

By Property \ref{item:MV}  of the appendix \ref{sec:reidemeister-torsion}, we deduce from the previous computations that 
$$ \bigl( \det p \bigr) \bigl( \det \psi \bigr) = \tau ( X, \op{Ad} \rho ) \prod_{i=1}^{n}   \Delta^2 ( \rho (C_i) ^{r_i} ) 
$$
which concludes the proof. 
\end{proof}

Using the duality between homology and cohomology and Poincar\'e duality, we have 
\begin{xalignat}{2} \label{eq:iso}
\begin{split}
 \op{det} H_\bullet ( X, \op{Ad} \rho) \simeq &  \op{det} ( H^1 ( X, \op{Ad} \rho ))  \otimes \bigl ( \op{det} ( H^2 ( X, \op{Ad} \rho)) \bigr)^{-1} \\
\simeq & \bigl( \op{det} ( H^1 ( X, \op{Ad} \rho )) \bigr)^{2}
\end{split}
\end{xalignat}

So  $(\det \psi )^{-1}$ may be viewed as the square of a volume element of $H^1 ( X, \op{Ad} \rho)$. Recall that $g^*$ induces an isomorphism from $H^1 (X, \op{Ad} \rho)$ to a symplectic vector space $(K, \Om)$. The following lemma is an easy  consequence of Theorem \ref{theo:symplectomorphism}.

\begin{lem} \label{lem:Liouville_det}
$g^*$ sends $( \det \psi )^{-1}$ to the square of the Liouville form $\Om^N/ N!$, where $N =\frac{1}{2} \dim K$. 
\end{lem}

\section{Application to moduli spaces} \label{sec:appl-moduli-spac}

Let us apply the previous results to the character manifold $\moir  (X)$. Recall that a density of a $n$ dimensional manifold $M$ is a section of the line bundle, whose fiber at $x$ is the space of applications $f: (T_x M)^n \rightarrow \R$ satisfying $f ( Ax_1, \ldots , Ax_n ) = | \det A | f(x_1, \ldots, x_n)$ for any endomorphism $A$ of $T_x M$. Here, we have natural densities on $\moir (X)$ and $\moir ( \Si, u)$ defined as follows: 

\begin{itemize} 
\item Since the tangent space $T_{[\rho]} \moir (X)$ is $H^1 (X, \op{Ad} \rho)$, by the isomorphism (\ref{eq:iso}), the torsion $\tau ( \op{Ad} \rho)$ is the inverse of the square of a density of $ T_{[\rho]} \moir (X)$. This defines a density $\mu_X$ of $\moir (X)$ whose value at $[\rho]$ is $\tau ( \op{Ad} \rho)^{-1/2}$. 
\item For any $u \in \Co (G)^n$, $\moir ( \Si , u)$ is a symplectic manifold, so it has a canonical density $\mu_u$. The symplectic structure of $T_{\rho} \moir ( \Si, u)$ is the form  $\Om$ considered in Lemma \ref{lem:symp_Structure} and if $N = \tfrac{1}{2} \dim \moir ( \Si, u )$, $\mu_u ( [\rho]) = | \Om^{\wedge N} | / N!$. 
\end{itemize}
For any $(u,v ) \in \mathcal{P}$, we defined a diffeomorphism $R_{u.v}$ from $\moir (X , u, v) $ to $ \moir ( \Si , u)$. By the proof of Lemma \ref{lem:tangent}, the linear tangent map of $R_{u,v}$ at $[\rho]$ is the map $g ^* : H^1 ( X, \op{Ad} \rho ) \rightarrow K$.
We deduce from Theorem \ref{theo:torsion} and Theorem \ref{theo:symplectomorphism} via Lemma \ref{lem:Liouville_det} our main result. 

\begin{theo} For any $(u,v) \in \mathcal{P}$, we have on $\moir (X , u, v)$
$$ \mu_X =  \Biggl( \prod_{i=1}^{n} \frac{  \Delta ( u_i ^{r_i} )  }{p_i^{\dim V_i /2}} \Biggr)   R_{u,v} ^* \mu_u $$
with $r_i$ any inverse of $q_i$ modulo $p_i$.
\end{theo} 

\section{Abelian case} \label{sec:abelian-case}

In this section, we adapt the previous result to the constant coefficient case. We consider the same Seifert manifold $X$ as above and we assume that the Euler number $$\chi = - \sum_{i=1}^{n} \frac{ q_i}{p_i} $$ does not vanish. 

For $Y = X, \Si, \partial \Si$, we let $H_j (Y) := H_j ( Y, \R)$. In contrast to the previous case, the groups $H_0 (X)$ and $H_3 (X)$ do not vanish.  
Introduce the three maps 
$$ f : H_1 ( \partial \Si ) \rightarrow H_1 ( \Si ), \quad g: H_1  (\Si) \rightarrow H_1 ( X ), \quad h: H_1 ( \Si ) \rightarrow H_2 (X)$$
defined as follows. $f$ and $g$ are the morphisms corresponding to the inclusions $\partial \Si \subset \Si$ and $\Si \subset X$ respectively. $h$ sends $\ga \in H_1 ( \Si)$ to the image of $\ga \boxtimes [S^1] \in H_2 (\Si \times S^1)$ in $H_2 ( X)$. 

\begin{prop} \label{prop:morphism_g_h}
$g$ and $h$ are surjective and their kernel is the image of $f$. \end{prop}

\begin{proof} 
The proof is similar to the one of Theorem \ref{theo:homologie_Seifert}, with the additional difficulty that $H_0 ( X ) \simeq \R$, $H_ 0( \Si)\simeq \R$ and $H_3 (X)\simeq \R$. The Mayer-Vietoris long exact sequence splits into three exact sequences: 
\begin{gather} \label{eq:mayer_vietoris} 
\begin{split} 
0 \rightarrow \R \rightarrow \R^n \xrightarrow{f} H_1 ( \Si) \xrightarrow{h} H_2 ( X) \rightarrow 0  \\
0 \rightarrow \R^{2n} \xrightarrow{B} \R^n \oplus H_1 ( \Si ) \oplus \R \xrightarrow{A} H_1(X) \rightarrow 0 \\
0 \rightarrow \R^n \xrightarrow{C} \R^n \oplus \R \rightarrow \R \rightarrow 0 
\end{split}
\end{gather}
where $C$ and $B$ are given by $C( x) = (x, x_1 + \ldots + x_n)$ and 
$B ( x,y) = ( z, f(x), y_1 + \ldots + y_n)$ with $z \in \R^n$ given by $z_i =  q_i x_i +p_i y_i$. By the first sequence in (\ref{eq:mayer_vietoris}), $h$ is surjective and its kernel is the image of $f$.
One checks that $\op{Im} B$ and $0 \oplus H_1 ( \Si) \oplus 0 $ are transversal subspaces, their intersection being $0 \oplus \op{Im} f \oplus 0$. Using that for any $\ga \in H_1 ( \Si)$,  $A ( 0, \gamma , 0 ) = g( \gamma)$, one deduces from the second sequence of (\ref{eq:mayer_vietoris}) that $g$ is surjective with kernel the image of $f$. 
\end{proof}
Let $\overline{\Si}$ be the closed surface obtained by gluing a disc to each boundary component of $\Si$. The inclusion $\Si \subset \overline{\Si}$ induces an isomorphism $H_1 ( \overline{\Si} ) \simeq H_1 ( \Si ) / \op{Im} f$. So by Proposition \ref{prop:morphism_g_h}, we have two isomorphisms 
$$\tilde{g} : H_1 ( \overline{\Si}) \rightarrow H_1 ( X), \qquad \tilde{h} : H_1 ( \overline{\Si}) \rightarrow H_2 ( X).$$ 

\begin{prop} \label{prop:abelian-poincare}
For any $\al \in H^1 (X)$ and $\be \in H^2 (X)$, we have $$\al \cdot_X \be = (\tilde{g}^* \al) \cdot_{\overline{\Sigma}} ( \tilde{h}^* \be ) $$ where $\cdot_X$ and $\cdot_{\overline{\Sigma}}$  denote the Poincar\'e pairings of $X$ and $\overline{\Si}$ respectively. 
\end{prop}

\begin{proof} 
The proof is similar to the one of Theorem \ref{theo:symplectomorphism}. First, since the image of $B$ in (\ref{eq:mayer_vietoris}) contains $0 \oplus 0 \oplus H_0 ( \Si)$, the image of $H^2( X) \rightarrow H^1 ( \Si \times S^1) \simeq H^1 ( \Si ) \oplus H^0 ( \Si)$ is contained in $H^1 ( \Si)$. Consequently  any class $\al$ of $H^1 (X)$ has a representative $a \in \Om ^1(X)$ such that
$$ a = p^* \tilde a  \qquad \text{  on $\Si \times S^1$ } $$ 
with $p : \Si \times S^1 \rightarrow \overline{\Si}$ the projection and $\tilde a \in \Om^1 ( \overline {\Si})$ a representative of $\tilde{g} ^* \al$.    Second, any class $\be \in H^2 (X)$ has a representative $b \in \Om^2 ( X)$ whose support is contained in an open subset of $\Si \times S^1$ and such that 
$$ b = p^* \tilde{b} \wedge q^* \tau + d \gamma$$
where $\tilde{b} \in \Om^1 ( \overline{\Si} )$ is a representative of $ \tilde{h}^* \be$, supported in an open subset of $\Si$, $q$ is the projection $\Si \times S^1 \rightarrow S^1$, $\tau \in \Om^1 (S^1)$ is such that $\int_{S^1} \tau = 1$ and $\ga \in \Om^1 ( \Sigma \times S^1)$. Finally, one checks that
$$ \int_{X} a \wedge b = \int_{\overline{\Si}} \tilde{a} \wedge \tilde{b}.$$
using Stokes' formula. 
\end{proof}
We can also compute the torsion of $X$ as in Theorem \ref{theo:torsion}. Since $H_0(X)$ and $H_3(X)$ have rank one, the torsion belongs to $H_0(X) \otimes \bigl(\det H_1(X) \bigr)^{-1} \otimes \det H_2 (X) \otimes \bigl(H_3 (X))^{-1}$.  
 
\begin{prop} \label{prop:abelian-torsion}
The Reidemeister torsion of $X$ is given by 
$$ \tau (X) =  \chi \prod_{i=1}^{n} p_i \;  [x] \otimes \det \psi \otimes  [X]^{-1}$$ 
where $\chi = - \sum q_i / p_i$ is the Euler number of $X$, $x \in X$ and $[x] \in H_0 (X)$ is the corresponding class,
$\psi $ is the map $\tilde{h} \circ \tilde{g}^{-1} : H_1 (X) \rightarrow H_2(X)$ and $[X] \in H_3(X)$ is the fundamental class. 
\end{prop}

\begin{proof} We adapt the proof of Theorem \ref{theo:torsion}.  Let $e=1 \in \R$, $(e_i)$ be the canonical basis of $\R^n$, $\delta = e_1 \wedge \ldots \wedge e_n$, $\rho = f(e_1) \wedge \ldots \wedge f( e_{n-1})$ and $\si \in \wedge^{2g} H_1( \Si)$ such that $\rho \wedge \si$ is a generator of $\wedge ^{2g + n -1 } H_1 ( \Si)$. Then one checks that the isomorphisms corresponding to the three exact sequences in (\ref{eq:mayer_vietoris}) send 1 into $e \otimes \delta ^{-1} \otimes ( \rho \wedge \si) \otimes h( \si) ^{-1}$, $ \chi^{-1} ( \prod p_i )^{-1} e \otimes ( \delta \otimes ( \rho \wedge \si) \otimes e)^{-1}\otimes h( \si)^{-1}$ and $ \delta^{-1}\otimes ( \delta \otimes e ) \otimes e^{-1}$ respectively. The factor $\chi (\prod p_i )$ appears because $B( \delta \otimes \delta) =  \chi (\prod p_i ) \delta \otimes \rho \otimes e$. The torsions of $\partial \Si \times S^1$,  $\partial \Si \times D$  and $\Si \times S^1$ are respectively $\delta \otimes ( \delta \otimes  \delta)^{-1} \otimes \delta$, $\delta \otimes \delta^{-1}$ and $( \rho \wedge \si) \otimes (\delta \otimes ( \rho \wedge \si) \otimes e)^{-1} \otimes e$. We conclude with Property \ref{item:MV} of appendix \ref{sec:reidemeister-torsion}. 
\end{proof}

Trivialising $H_0 (X)$ and $H_3(X)$ by sending $[x]$ and $[X]$ to 1,
and identifying $H_1(X)$ with the dual of $H_2 (X)$ by Poincar\'e
duality, the inverse of square root of the torsion gets identified
with an element of $\det H_1 (X)$. By propositions
\ref{prop:abelian-poincare} and \ref{prop:abelian-torsion}, the
torsion satisfies
\begin{gather} \label{eq:torsion_Liouville_abelien} 
 \tilde{g} ^* ( \tau (X) )^{-1/2} = \Bigl| \chi \prod_{i=1}^{n} p_i \Bigr|^{-1/2}  \mu
\end{gather}
where $\mu \in \det H_1 ( \overline{\Si})$ is the Liouville density of $H^1 ( \overline{ \Si})$. 

This may be applied to the space $\jac (X)$ consisting of representation of $\pi_1 (X)$ in $\uun$ as follows. First, for any connected compact manifold $Y$,  $\jac (Y)$  is an abelian Lie group, the product being the pointwise multiplication.  The Lie algebra of $\jac (Y)$ is the space of morphisms from $\pi_1(Y)$ to $\R$, which identifies with $H^1 (Y)$.  In particular for the Seifert manifold $X$, the Lie algebra of $\jac (X)$ being $H^1(X)$, $(\tau (X))^{-1/2}$ determines an invariant density of $\jac (X)$. Furthermore, the Lie algebra of $\jac (\overline{\Si})$ being $H^1( \overline{\Si})$, $\jac (\overline{\Si})$ has an invariant symplectic structure and a corresponding Liouville density.  

For any $(u,v) \in \uun ^{n+1}$, let $\jac (X, u,v)$ be the subset of $\jac ( X)$ consisting of the representations $\rho$ such that $\rho (C_i) = u_i$ for any $i$ and $\rho ( S^1) = v$. Then 
\begin{gather} \label{eq:dec_pi0}
 \jac ( X) = \bigcup_{(u,v ) \in \mathcal{Q}} \jac ( X, u, v ) 
\end{gather}
where $\mathcal{Q}$ is the set of $(u,v) \in \uun ^{n+1}$ such that $u_1 \ldots u_n =1$ and for any $i$, $u_i ^{p_i} =v^{q_i} $. Since the Euler number $\chi$ does not vanish, $\mathcal{Q}$ is finite. 
Furthermore, for any $(u,v) \in \mathcal{Q}$, $\jac (X,u,v)$ is connected. So (\ref{eq:dec_pi0}) is the decomposition of $\jac (X)$ into connected components. 

Let $\mathbf 1 = ( 1,\ldots , 1) \in \uun$. $\jac (X, \mathbf{1}, 1)$ is the component of the identity of $\jac (X)$. We have a natural Lie group isomorphism $\Phi$ from $\jac (X, \mathbf{1}, 1)$ to $\jac( \overline{\Si})$, such that for any $\rho \in \jac( X, \mathbf{1},1)$, the restrictions of $\rho$ and $\Phi ( \rho)$ to $\Si$ are the same. The linear tangent map at the identity to $\Phi$ is the adjoint map to the map $\tilde{g} : H_1 ( \overline{\Si}) \rightarrow H_1 ( X)$. Thus Equation (\ref{eq:torsion_Liouville_abelien}) computes the invariant density of $\jac ( X, \mathbf{1}, 1)$ in terms of the pull back by $\Phi$ of the Liouville density. We recover in this way  Theorem 9 of \cite{McLe}.

\appendix

\section{Representation space} \label{sec:representation-space}

The general theory describing the smooth structure of a representation space is rather involved and belongs more to algebraic geometry, \cite{LuMa}. In this appendix, we summarize the basic general facts we need, remaining in the context of differential geometry. 

Let $G$ be a connected Lie group and $\pi$ be a finitely generated group. 
Let $\rep ( \pi)$ be the space of representations of $\pi$ in $G$. For any set of generators $a = (a_1, \ldots , a_N)$ of $\pi$, the map  
$$\xi_a : \rep ( \pi) \rightarrow G^N, \qquad \rho \rightarrow (\rho (a_1), \ldots, \rho (a_N)),$$
 is injective, and allows us to identify $\rep ( \pi)$ with $\xi_a( \rep( \pi))$. If $a= ( a_1, \ldots , a_N)$ and $b = (b_1, \ldots, b_M)$ are two  sets of generators, the bijection $\xi_b \circ \xi_a^ {-1}$  from $\xi_a (\rep ( \pi))$ onto $\xi_b ( \rep  ( \pi))$ is a homeomorphism. Indeed,  expressing the $a_i$'s in terms of the $b_j$'s, we obtain a smooth map $\varphi : G^N \rightarrow G^M$ extending $\xi_b  \circ \xi_a ^{-1}$. We endow $\rep ( \pi)$ with the topology such that for any set of generators $a$ of $\pi$, $\xi_a$ is a homeomorphism onto its image.

Let $\rep ^{\op{s}} ( \pi)$ be the set of representations $\rho$ of $\pi$ in $G$ admitting an open neighborhood $U$ and a set of generators $(a_1, \ldots , a_N)$ such that $\xi_a ( U)$ is a smooth submanifold of $ G^N$. $\rep ^{\op{s}} (\pi)$ has a unique manifold structure such that for any such pairs $(U, a)$, the map $\xi_a : U \rightarrow \xi_a ( U)$ is a diffeomorphism. Indeed, arguing as above, we see that for any two pairs $(U, a)$ and $(V,b)$ the map $\xi_b \circ \xi_a ^{-1} : \xi_a (U\cap V) \rightarrow \xi_b (U \cap V)$ is a diffeomorphism.  

For any representation $\rho$ of $\pi$ in $G$, composing $\rho$ with the adjoint representation, the Lie algebra $ \mathfrak{g}$ becomes a left $G$-module. Consider the corresponding cochain complex in degrees 0 and 1: $C^0 ( \pi, \op{Ad} \rho) = \mathfrak{g}$, $C^1 ( \pi, \op{Ad \rho}) = \op{Map} ( \pi, \mathfrak{g})$, the differential in degree 0 is  
\begin{gather*} \label{eq:def_drho}
d_{\rho} : \mathfrak{g} \rightarrow C^1(\pi, \op{Ad}  \rho),   \quad \xi \rightarrow \bigl( \ga \rightarrow  \op{Ad}_{\rho ( \ga)} \xi - \xi  \bigr) 
\end{gather*}
and the space of 1-cocycle
$$ Z^1 ( \pi, \op{Ad} \rho ) = \bigr\{ \tau : \pi \rightarrow \mathfrak{g}; \;  \forall \ga_1, \ga_2 \in \pi, \; \tau ( \ga_1 \ga_2) = \tau ( \ga_1 ) + \op{Ad}_{\rho ( \ga_1)} \tau ( \ga_2) \bigl\}.$$

For any $ \ga \in \pi$, the map $e_{\ga} : \rep ( \pi ) \rightarrow G$ sending $\rho$ into $\rho ( \ga)$ is continuous. Its restriction to $\rep ^{\op{s}} ( \pi)$ is smooth. 
If $\rho \in \rep ^{\op{s}} ( \pi )$, we have a natural map from $T_{\rho} \rep ^{\op{s}} ( \pi)$ to $Z^1 ( \pi, \op{Ad} \rho )$ sending $\dot{\rho}$ to the cocycle $\tau$ given by 
$$ \tau ( \ga) = R_{ \rho( \ga) ^{-1}} T_{\rho} e_\ga (\dot{ \rho}), \qquad \forall \ga \in \pi,$$
where for any $g \in G$, $R_{g^{-1}} : T_{g} G \rightarrow \mathfrak{g}$ is the linear map tangent to the right multiplication by $g^{-1}$. It is easily seen that this map is well-defined and injective, so we consider the tangent space $T_{\rho} \rep ^{\op{s}} ( \pi)$ as a subspace of $Z^1 ( \pi, \op{Ad} \rho )$.

$G$ acts on $\rep ( \pi)$ by conjugation. The action preserves $\rep ^{\op{s}} ( \pi)$. A straightforward computation shows that the infinitesimal action at $\rho \in \rep^{\op{s}} ( \pi)$ is the differential $d_{\rho}$ introduced above.

Assume from now on that $G$ is compact. The subset $\rep^0 ( \pi)$ of $\rep ( \pi)$ consisting of irreducible representations is open. Indeed, if $(a_1, \ldots , a_n)$ is any set of generators, then $\xi_a ( \rep^0 ( \pi) ) = \xi_a ( \rep ( \pi) ) \cap (G^N)^0$ where $(G^N)^0$ consists of the $N$-uplets whose centralizer in $G$ is the center.  By the slice theorem for action of compact Lie group, $(G^N)^0$ is open in $G^N$, because it is either empty or the principal stratum for the diagonal action of $G$ on $G^N$ by conjugation.

Set $\rep ^{\op{s}, 0} ( \pi) := \rep ^0 ( \pi ) \cap \rep ^{\op{s}} ( \pi)$. The quotient space $\mo ^{\op{s},0} ( \pi) := \rep ^{\op{s}, 0} ( \pi) / G$ is a smooth manifold because it is the quotient of a smooth manifold by a smooth action of the compact Lie group $G$ with constant isotropy $Z(G)$. Furthermore,  for any $\rho \in \rep ^{\op{s}, 0} ( \pi)$,  the infinitesimal action at $\rho$ being $d_{\rho}$, we have 
$$ H^0 ( \pi , \op{Ad} \rho ) = \ker d_{\rho} = \mathfrak{z} ( \mathfrak{g}) $$
where $\mathfrak{z} ( \mathfrak{g})$ is the Lie algebra of the center of $G$. Furthermore $T_{[\rho]}  \mo ^{\op{s},0} ( \pi)$ identifies with a subspace of $H^1 ( \pi, \op{Ad} \rho) = Z^1 ( \pi , \op{Ad} \rho) / \op{Im} d _{\rho}$. 

\section{Reidemeister Torsion} \label{sec:reidemeister-torsion}

Let $M$ be a compact manifold possibly with a non empty boundary. Let $E \rightarrow M$ be a flat real vector bundle equipped with a flat metric. Denote by $ \det H_{\bullet} (E)$ the line
$$\det H_{\bullet} (E) =  \det H_0 (E) \otimes (\det H_1 (E))^{-1} \otimes \ldots \otimes ( \det H_{n} (E))^{(-1)^n} $$
where $n$ is the dimension of $M$ and for any finite dimensional vector space $V$, $\det V = \bigwedge ^{\op{top}} V$. In the acyclic case, $\det H_{\bullet} (E) = \R$.   
The Reidemeister torsion of $E$ is a non-vanishing vector 
$ \tau (  E) \in \det H_{\bullet} ( E)  $ well-defined up to sign. Let us recall briefly its definition. 

Let $K$ be the simplicial complex of a smooth triangulation of $X$. For any cell $\si$ of $K$, let $E_\si$ be the space of flat sections of the restriction of $E$ to $\si$.  Introduce the complex $C_\bullet (K,E)$ where $C_k (K,E) = \bigoplus_{ \dim \si = k} E_{\si}$ with the usual differential. Then the $H_k (E)$ are the homology groups of $C_{\bullet} (K,E)$. Consequently, we have an isomorphism $\det C_{\bullet}(K,E) \simeq \det H_{\bullet} (E)$. Furthermore, for any cell $\si$, $E_\si$ is an Euclidean space. So $C_k (K, E)$ has a natural scalar product where the $E_\si$ are mutually orthogonal, and $\det H_{\bullet} (E)$ inherits an Euclidean product by the previous isomorphism. The Reidemeister torsion $\tau (E)$ is by definition a unit vector of $\det H_{\bullet} ( E) $. It does not depend on the choice of the triangulation, cf. \cite{Mi}, Section 9. 

The torsion satisfies the following properties, cf. \cite{KwSz} for \ref{item:somme}, \ref{item:prod} and \cite{Mi}, Section 3 for \ref{item:MV}. 
\begin{enumerate} 
\item \label{item:somme}
Let $E = E_1 \oplus E_2$ where $E_1$ and $E_2$ are two flat Euclidean vector bundles with base $M$. Then we have a natural isomorphism $H_\bullet (E) \simeq H_\bullet (E_1) \oplus H_\bullet (E_2)$. The corresponding isomorphism $\det H_\bullet (E) \simeq \det H_\bullet (E_1) \otimes \det H_\bullet (E_2)$ sends $\tau (E)$ into $\tau (E_1) \otimes \tau (E_2)$. 

\item \label{item:prod}
Let $E_1 \rightarrow M_1$ and $E_2 \rightarrow M_2$ be two flat Euclidean vector bundles. Assume that $M_1$ is closed. Set $M = M_1 \times M_2$ and $E = E_1 \boxtimes  E_2 $. By K{\"u}nneth theorem,  we have $H_\bullet (E) \simeq H_\bullet (E_1)\otimes H_{\bullet} (E_2)$. The corresponding isomorphism 
$$\det H_\bullet (E) \simeq \bigl( \det H_\bullet (E_1) \bigr)^{\chi(E_2)} \otimes \bigl( \det H_\bullet (E_2) \bigr)^{\chi(E_1)}$$ sends $\tau (E)$ into $ \tau (E_1) ^{\chi (E_2)} \otimes  \tau (E_2)^{\chi(E_1)}$. 

\item  \label{item:MV}
Let $E$ be a flat Euclidean vector bundle whose base $M$  is obtained by gluing two manifolds $M_1$, $M_2$ along their boundary $ N$. By the Mayer-Vietoris exact sequence, we have an isomorphism
$$ \det H_{\bullet} ( E ) \otimes \det H_{\bullet} ( E|_N) \simeq \det  H_{\bullet} ( E|_{M_1} ) \otimes \det H_\bullet (E|_{M_2}) .
$$
This isomorphism sends $ \tau (E) \otimes \tau (E|_{N}) $ to $\tau ( E|_{M_1} ) \otimes \tau ( E|_{M_2} )$. 
\end{enumerate}
Finally, it is a classical exercise to compute the torsion of a bundle over a circle. 
\begin{enumerate} \setcounter{enumi}{3}
\item \label{item:cercle}
Let $E$ be a flat Euclidean vector bundle $E$ on an oriented circle $C$. Let $p \in C$  and let $\varphi: E_p \rightarrow E_p$ be the holonomy of $C$. Let $H = \ker ( \varphi - \op{id} ) $. We have two isomorphisms $H_0 ( E) \simeq H$ and $ H_1 (E) \simeq H $
sending $u \in H$ into $[p] \otimes u$ and $[C] \otimes u$ respectively. Thus $\det H_\bullet (E) \simeq \det H \otimes ( \det H )^{-1} \simeq \R$, so that the torsion may be considered as a real number. With this convention, we 
\begin{gather*} 
   \tau (E) =  {\det }^{-1} \bigl( (\varphi - \op{id} )|_{H^{\perp}} \bigr)
\end{gather*}
where $H^{\perp}$ is the orthogonal complement of $H$. 
\end{enumerate}

\bibliographystyle{alpha}
\bibliography{bibtorsion}

\Addresses

\end{document}